\documentclass[11pt,english]{amsart}

\usepackage{amssymb,amsthm, amsmath, amsfonts, url, graphicx,verbatim,scrextend}
\usepackage{amscd}
\usepackage{enumitem}
\usepackage[utf8]{inputenc}
\usepackage[english]{babel}
\usepackage{lipsum}
\usepackage{bm}
\theoremstyle{plain}
\newtheorem*{thm*}{Theorem}
\theoremstyle{plain}
\newtheorem{thm}{Theorem}[section]
\theoremstyle{definition}
\newtheorem{defn}[thm]{Definition}
\theoremstyle{plain}
\newtheorem{lem}[thm]{Lemma}
\theoremstyle{plain}
\newtheorem{prop}[thm]{Proposition}
\theoremstyle{plain}
\newtheorem{cor}[thm]{Corollary}
\theoremstyle{remark}

\theoremstyle{remark}
\newtheorem*{acknowledgement*}{Acknowledgement}

\title{Algebraic conditions for the positivity of sectional curvature}
\author{Dan Gregorian Fodor}
\email{dan.fodor52@yahoo.com}
\address{Faculty of Mathematics, Alexandru Ioan Cuza University, Ia\c si, Romania }
\date{}

\begin{document}

\pagenumbering{roman}

\pagestyle{myheadings}\markboth{}{}

\pagenumbering{arabic}
\pagestyle{myheadings}

\maketitle

\begin{abstract}
We examine algebraic conditions for the sectional positivity of the Riemann curvature operator. We describe sufficient conditions for dimension $n=4$, and complete characterisation for a dense open subset of the space of operators in dimension $4$. We briefly examine higher-dimensional curvature operators.
\end{abstract}

\section{Introduction}
Positive sectional curvature is an area of much interest. Positive curvature operators are easy to characterise algebraically: the critical values of a curvature operator are the roots of its characteristic polynomial, thus, the operator is positive if and only  its coefficients are of alternating sign. In this paper we apply methods of algebraic geometry to obtain an object analogous to the "characteristic polynomial": for a dense open subset of the set of  $4$-curvature operators, the real roots of our polynomial are exactly the critical values of sectional curvature. Our polynomial is obtained by applying the discriminant to a specific algebraic object constructed from the curvature operator. Similar results to ours have recently been obtained in  \cite{New}, Theorem C.

The \emph{discriminant} of a polynomial $p$ of degree $d$ is an algebraic expression of its coefficients that gives zero if and only if the polynomial has a multiple complex root. It is equal to the determinant of the Sylvester matrix of $p$ and $p'$ (\cite[Chapter 4]{A1}). It is important to note that when we set the top coefficient of our degree-$d$ polynomial to $0$, the value of the degree-$d$ discriminant expression also becomes 0. We do not recover the expression of the $d-1$ discriminant in its remaining coefficients.

We use $ \mathrm{disc_y}(p(x,y))$  to denote the discriminant of the algebraic expression   $p(x,y)$, with respect to to the free variable $y$. Since we are working in a single point with fixed metric, we do not explicitly write out the indices of our tensors. The metric gives an equivalence between covariant and contravatiant indices, allowing us to always obtain geometrically sound quantities.

The critical values of sectional curvature can be characterised in terms of a Lagrangian.  This method has be previously applied by Thorpe (see \cite{J},\cite{J2}), Singer (see \cite{ST}, section 2), and P{\"u}ttmann (see \cite{P}, Section 3.1). We construct the Lagrangian for $4$-curvature operators in Proposition \ref{prop1} and for higher dimensional curvature operators  in Section 4. For the $4$-dimensional case, we obtain  $L_R(v,x,y) = vRv - x\cdot (vIv-1) - y \cdot (vKv)$ , where $K$ is the $4$-volume form, and $I$ is the identify operator on the set of $2$-forms. The critical values of sectional curvature are the $x$-components of the critical points $(v,x,y)$ of our Lagrangian. We define $(x,y)$ as a \emph{critical point} of the operator $R$ if there exists a $v$ making $(v,x,y)$ a critical point for our Lagrangian (see Definition \ref{d1}). By Theorem \ref{thm1}, we give a characterisation of such points purely in terms of the coefficients of $p(x,y)= \det(R-xI-yK)$, thus eliminating reference to $v$ and making the switch from a geometric to an algebraic viewpoint. By Theorem \ref{tx}, we show that the set of real roots of $q(x) =\mathrm{disc_y}(p(x,y))$ include the set of critical values of sectional curvature. By Theorem \ref{nt}, we show that when $\mathrm{disc(q)}\neq 0$, the two sets coincide. It is in principle possible, by the Tarski–Seidenberg theorem \cite{A1}, to generate a complete quantifier-free description of the conditions for positive sectional curvature. However, such descriptions tend to be long, unstructured, and difficult to compute. Our result trades completeness, (necessity only works on a dense open subset of the space of curvature operators) for simplicity.
Section $4$ examines the possibility of an algebraic approach to higher dimensional curvature operators, and shows the failure of a naive generalization of our theorem. Section $5$ applies algebraic methods and results from earlier sections to recover results form \cite{J}.

\begin{thm}\label{main} {\bf (Main result)}~ \\
Let $n =4$ and consider a given Riemann tensor $R$ as an operator on $\Lambda^2$, the $6$-space of $2$-forms. Let $I$ be the identity operator on $\Lambda^2$, and $K$ be the $4$-volume form with $2$ indices lifted. \\ Let $p(x,y) = \det(R-xI-yK)$, and $q(x) = \mathrm{disc_y}(p(x,y))$. \\ The set of real roots of $q$ includes the set of critical values of the sectional curvature of $R$. If $\mathrm{disc_x}(q(x))\neq 0$ the two sets coincide.\\ As such, if all real roots of $q$ are positive (nonnegative), then $R$ is sectionally positive (nonnegative).
When $\mathrm{disc_x}(q(x))\neq 0$, the converse also holds: $R$ is sectionally positive (nonnegative) if and only if all real roots of $q$ are positive (nonnegative).

\end{thm}

\section{Preliminary notions}
Here we present some lemmas and definitions that will be used trough the rest of the article. First we shall study the algebraic conditions for the positivity of the curvature operator.
\begin{thm} \label{t_1}
A real symmetric matrix $M$ is positive definite if and only if the coefficients of its characteristic polynomial are of alternating sign. It is negative definite if and only if its coefficients are positive.
\end{thm}
\begin{proof}
A matrix is positive-definite if and only if $vMv > 0, \forall v \neq 0$. Due to $M$ being real symmetric, $M$ is positive definite if and only if all its eigenvalues are positive: $\forall k\in\{1..n\},\:\lambda_{k}>0$. The characteristic polynomial $T_{M}$ satisfies:
$$\det(Ix-M) =  \sum_{k=0}^{n} (-1)^{n-k} e_{k}(\lambda_1,\lambda_2,..\lambda_n) x^k $$
where $e_k$ denotes the $k$'th elementary symmetric polynomial. We know that if the eigenvalues are positive, then the elementary symmetric polynomials are positive, and the coefficients of $T_{M}$ are of alternating sign. For the converse, we use the coefficients of $T_M$ to construct it from its Taylor series. We may prove inductively that a polynomial with coefficients of alternating sign is monotonous and of constant sign on the $(-\infty,0]$ interval. Therefore, all its real roots must be in the  $(0,\infty]$ interval. As the roots of $T_M$ are the eigenvalues $\lambda_{k}$, these must be positive. .

For negative-definiteness, observe that $M$ is positive definite if and only if $-M$ is negative definite. Knowing the relationship between $M$ and $-M$'s eigenvalues and applying the above result to $-M$, we see that $M$ is negative-definite if and only if the coefficients of its characteristic polynomial are positive. This completes our proof.
\end{proof}
\begin{cor}
A curvature tensor $R$, taken as an operator on the $\frac{n(n-1)}{2}$-dimensional space of $2$-forms $\Lambda_{2}$, is positive definite if and only if the coefficients of  $T_{R}(x) = \det(Ix-R)$ are of alternating sign.
\end{cor}

\noindent We define a $2$-form $W$ as being decomposable if and only if there exist $A, B,$   $1$-forms, such that $A \wedge B = W$.

\begin{lem}
A $2$-from $W$ is decomposable if and only if $W \wedge W = 0$.
\end{lem}
\begin{proof}
Treating $W$ as a skew-symmetric matrix, it admits a decomposition:
$$W = \sum_{k=1}^{m} \lambda_k  (V_{2k-1} \wedge V_{2k})$$
where $\forall k,\lambda_k \neq 0$, and $V_0, V_1, \cdots, V_{2m}$ are a set of orthonormal $1$-forms. \\
Using this decomposition, we obtain:
$$W \wedge W = \sum^{m}_{a < b} 2\lambda_a \lambda_b(V_{2a-1} \wedge V_{2a} \wedge V_{2b-1} \wedge V_{2b})$$
Thus $W\wedge W=0$ if and only if $W$ is decomposable.
\end{proof}
\begin{lem}\label{vform}
Let $K_1, K_2, \cdots, K_{\binom{n}{4}}$ be an orthonormal basis for the set of $4$-forms in dimension $n$. A $2$-from $W$ is decomposable if and only if $$\forall p\in\{1,2,\cdots \binom{n}{4}\},\:\, W_{ab} (K_p)^{abcd} W_{cd} = 0$$ In particular, for $n=4$, $W$ is decomposable if and only if $W_{ab} K^{abcd} W_{cd} = 0$, where $K$ is the $4$-volume form.
\end{lem}
\begin{proof}
Using the previous lemma, $W$ is decomposable if and only if $W\wedge W = 0$. However, $W\wedge W$ is the antisymmetric part of $W \otimes W$, ie. the projection of $W \otimes W$ onto the space of $4$-forms. Therefore, it is zero if and only if $$\forall p\in\{1,2,\cdots \binom{n}{4}\},\:\,  W_{ab} (K_p)^{abcd} W_{cd} = 0.$$
\end{proof}
\section {Curvature operators in dimension 4}
Let $S_2$ be the set of decomposable $2$-forms of norm $1$. This is the Grassmannian of oriented $2$-planes, and an affine variety in the unit sphere of the space of $2$-forms $\Lambda_{2}$. The curvature operator $R$ is normally taken as a quadratic form on the space of $2$-forms. We consider its restriction to $S_2$, given by $f:S_2 \to {\mathbb R}$, $f(v)=vRv$.

The curvature operator $R$ is \emph{ defined} to be sectionally positive if $f(v)>0, \forall v\in S_2$.

Due to the fact that $S_2$ is compact we obtain the following.  $R$ is sectionally positive if and only if $f(v)>0$, $\forall v\in S_2$ such that $v$ is a critical point for $f$. We shall study the critical points of $f$ on $S_2$, in dimension $4$, using Lagrange multipliers. This method is found in [Singer and Thorpe,\cite{ST}, Section 2] and [P{\"u}ttmann \cite{P}, Section 3.1].

Maximise $vRv$ subject to constraints $v \wedge v =0$ and $vIv = 1$, where $I$ is the norm on $2$-forms (it coincides with the curvature operator of constant curvature 1). In dimension $4$, $v \wedge v =0$ becomes $vKv = 0$, where $K$ is the $4$-volume form Lemma (\ref{vform}). The resulting Lagrangian function is
\begin{eqnarray} L_R(v,x,y) = vRv - x\cdot (vIv-1) - y \cdot (vKv). \label{lr}
\end{eqnarray}

Its gradient gives us the following conditions for critical points on $v$:
\begin{itemize}
\item[1)] $vIv = 1$;
\item[2)] $vKv = 0$;
\item[3)] $(R-xI-yK)v=0$.
\end{itemize}
Contracting the third equation with $v$ we obtain  $v(R-xI-yK)v=0$, which simplifies to $vRv = x$. The Lagrange multiplier allows us to reformulate our problem:
\begin{prop} \label{prop1} A $4$-curvature operator  $R$ is sectionally positive if and only if for all triples $(v,x,y)$ satisfying 1), 2), 3), we have $x>0$.
\end{prop}

The next step is to reformulate the problem in a way that does not make explicit reference to $v$. We shall use the following property to make the problem computationally easier: If $(v,x,y)$ corresponds to a critical point of the operator $L_R$, then $(v,0,0)$ corresponds to a critical point of the operator $L_{R'}$, where $R' = R - xI - yK$. The next lemma shall deduce when a symmetric operator $R$ has a vector $v$ satisfying:
\begin{itemize}
\item[1)] $vIv = 1$;
\item[2)] $vKv = 0$;
\item[3)] $Rv=0$.
\end{itemize}

\begin{lem} \label{lem1}
Let $R, K$ be symmetric operators. Let
\begin{eqnarray*}
\det(R-xI-yK) = \sum a_{mn}x^{m}y^{n} \\
P_k(x,y) = \sum_{m+n = k} a_{mn}x^{m}y^{n}
\end{eqnarray*}
The following properties hold:
\begin{itemize}
\item[1)] If $a_{00}, a_{10},\ldots,a_{k0} = 0$, then $P_k=0$.
\item[2)] There exists a $v \neq 0$ such that $Rv = 0$ and $vKv = 0$ if and only if, for the smallest $k$ such that $P_k \neq 0$, the coefficients of  $P_k(x,1)$ are neither all of the same sign, nor are they strictly alternating.
\end{itemize}
\end{lem}

\begin{proof}
For the proof we shall make use of the following lema \cite{D1}:
\begin{lem}
Let $A,B$ be $n$-square matrices. Then:
$$\det(A+B) = \sum_r \sum_{\alpha,\beta} (-1)^{s(\alpha)+s(\beta)} \det(A[\alpha|\beta]) \det(B(\alpha|\beta))$$
where the outer-sum $r$ is over the integers $0,\ldots,n$, the inner sum is over all strictly increasing integer sequences of length $r$ chosen from $1,\ldots,n$, $A[\alpha|\beta]$ is the $r$-square submatrix of $A$ lying in rows $\alpha$ and columns $\beta$,  $B(\alpha|\beta)$ is the $(n-r)$-square submatrix of $B$ lying in rows complementary to $\alpha$ and columns complementary to $\beta$, and $s(\alpha)$ is the sum of the integers in $\alpha$.
\end{lem}

Applying the above lemma to $A=R$, $B=(-xI-yK)$, and separating the terms by degrees, we obtain:
$$P_k(x,y) = \sum_{\alpha,\beta}^{|\alpha| =|\beta|=k} (-1)^{s(\alpha)+s(\beta)} \det(-(xI+yK)[\alpha|\beta])\det(R(\alpha|\beta))$$

1) $a_{00}, a_{10},\ldots, a_{k0}$ are the first $k+1$ coefficients of $R$'s characteristic polynomial. Due to $R$ being symmetric and the coefficients being equal to $0$, the dimension of $\ker(R)$ is at least $k+1$. We know  the terms $\det(R(\alpha|\beta))$ in $P_{k}(x; y)$ go over all $n-k$ minors of $R$. Due to the dimension of $\ker(R)$, all these terms equal $0$. Therefore $P_k(x,y) = 0$.

2) There exists such a $v$ if and only if the intersection of the null-cone of $K$ and $\ker(R)$ is non-trivial. This is equivalent to the cone of the restriction of $K$ to $\ker(R)$ being non-trivial. Denote $K'$ to be the restriction of $K$ to the $\ker(R)$. There exists a $v\neq 0$ such that $vK'v =0$ if and only if $K'$, is neither strictly positive nor strictly negative. This translates to the coefficients of the characteristic polynomial. Denote $T_{K'} (x)$ to be the characteristic polynomial of $K'$. There exists a $v\neq 0$ such that $vK'v = 0$ if and only if the coefficients of $T_{K'}$ are neither all strictly of the same sign, nor are they strictly alternating: if $vK'v = 0$ then $K'$ is neither positive definite nor negative definite. The conditions on the coefficients of $T_{K'}$ follow from Theorem \ref{t_1}. The converse also follows: if the coefficients of $T_{K'}$  are neither of the same sign nor strictly alternating, then $K'$ is neither positive definite nor negative definite. There exists $v_1,v_2 \neq 0$ such that $v_1K'v_1 \leq 0$ and $v_2K'v_2 \geq 0$. The existence of $v\neq 0$, $vK'v =0$ then follows from the intermediate value theorem.

Let $P_{k}$ be the first non-zero homogenous polynomial of $\det(R-xI-yK)$. We now prove that  $P_{k}(x,1) = a_{k0} T_{K'}(x)$:
Given that $P_{k}$ is the first non-zero homogenous polynomial, the dimension of the kernel of $R$ is $k$. We select an orthonormal basis for our matrices such that the kernel of $R$ is spanned by the first $k$ eigenvectors. This gives $m_{ab}=0, \forall a\in\{1..k\}, \forall b\in\{1..n\}$, where $m_{ab}$ are the matrix coefficients of $R$. Applying to the above formula of $P_k$, we obtain:
$$P_k(x,y) = \det(-(xI+yK)[\alpha|\alpha])\det(R(\alpha|\alpha))$$ where $\alpha= \{1..k\}$. We have $(R(\alpha|\alpha))$, ie the complementary minor of the first $k,k$ minor, to be the only $n-k,n-k$ minor with non-zero determinant, as all the other minors intersect $R$'s kernel. Setting $y=1$, and knowing that the first $k,k$ minor of $K$ is its restriction to the kernel of $R$ shows $P_{k}(x,1) = a_{k0} T_{K'}(x)$.

This proves our result regarding the coefficients of $P_k$.
\end{proof}
\begin{defn} \label{d1}
Given a $4$-curvature operator $R$ and $K$ the volume form,  a point  $(x_1,y_1) \in \mathbb{R}^2$ is a \emph{critical point} for $R$ if the coefficients of the first non-zero homogenous polynomial of  $p(x,y) = \det((R-x_{1}I-y_{1}K) -xI-yK)$ are neither all of the same sign, nor all of strictly alternating sign. Equivalently, $(x_1,y_1)$ is a critical point of $R$ if there exists a vector $v$ making $(v,x_1,y_1)$ a critical point for the Lagrangian $L_{R}$.
\end{defn}

Combining Proposition \ref{prop1} with the previous lemma, we further refine our criteria:

\begin{thm} \label{thm1}
A $4$-curvature operator $R$ is sectionally positive (nonnegative) if and only if for all critical points $(x_1,y_1)$, $x_1$ is positive (nonnegative).
\end{thm}

\begin{proof}
Denoting $R' = (R-x_{1}I-y_{1}K)$, according to \ref{lem1}, points satisfying the above conditions are exactly those pairs $(x_1,y_1)$ for which there exists a $v$ such that  $vIv = 1$, $vKv = 0$, and  $v(R')=0$. Back-substituting $R'$ as  $(R-x_{1}I-y_{1}K)$, we see, due to \ref{prop1}. that  $(v,x_1,y_1)$ are the critical points of $L_R$, thus completing the proof.
\end{proof}

This gives us a complete algebraic description of the positivity of $4$-curvature operators, having only two quantified variables, $x$ and $y$. The quantifiers can be eliminated with the Tarski–Seidenberg theorem \cite{A1}, however, the resulting description is long and unwieldy. We shall instead study a dense open subset of the set of operators, on which the criteria for positive sectional curvature become simpler to describe.

\begin{lem} \label{lem2}
Let $(x_1,y_1)$ be a critical point of a $4$-curvature operator $R$, and define $R' = R-x_{1}I-y_{1}K$ and $p(x,y) = \det(R' -xI-yK)$. Then, for the given $p$, we have $a_{00}=0$ and $a_{01}=0$.
\end{lem}
\begin{proof}
If $a_{00} \neq 0$ then the kernel of $R'$ would be of dimension zero, there would exist no $v$ that makes $(x_1,y_1)$ a critical point. For $a_{00} = 0$ there are two situations that must be considered: Either the kernel of $R'$ is of dimension $1$, or the kernel of $R'$ is of dimension greater than $1$.

Assume the kernel is of dimension $1$. Then the coefficients of $P_1(x, 1) = a_{10}x+ a_{01}$ must be neither of strictly alternating sign nor all of the same sign. Therefore   $a_{01} = 0$.  Assume the kernel is of dimension greater than 1. Then $a_{00} = 0, a_{10} = 0$. According to Lemma \ref{lem1}, $P_1=0$, therefore $a_{01} = 0$.
\end{proof}

\begin{thm} \label{tx}
Let $R$ be a $4$-curvature operator. If $(x_1,y_1)$ is a critical point of $R$, then for $q(x) = \mathrm{disc_y}(\det(R -xI-yK))$ we have $q(x_1) = 0$. Therefore, if all the real roots of $q$ are positive (nonnegative), then $R$ is sectionally positive (nonnegative).
\end{thm}
\begin{proof}
 If $(x_1,y_1)$ is a critical point, defining $R' = (R-x_{1}I-y_{1}K)$, and $p(x,y) = \det(R' -xI-yK)$, we have $a_{00} = 0$ and $a_{01} = 0$ for $p$ (see Lemma \ref{lem2}). Defining $p_1(x,y) = \det(R -xI-yK) = p(x-x_1,y-y_1)$, this means that $y_1$ is a double root in $p_1(x_1,y)$, taken as a polynomial of $y$. We have $p_1(x_1,y) = p(0,y) = \sum_{k=2..6} a_{0k}y^{k}$. This means that the discriminant in $y$ of $p_1(x,y) = \det(R -xI-yK)$ gives $0$ for $x=x_1$. Having  $q(x) = \mathrm{disc_y}(\det(R -xI-yK)) =  \mathrm{disc_y} (p_1(x,y))$, this completes the proof.
 \end{proof}

This gives us a sufficient condition for sectional positivity (non-negativity). However, it is stronger than necessary. Though all critical points $(x_1,y_1)$ of $R$ correspond to real roots $x_1$ of $q$, the converse is not true. The next step is studying the so-called "false" real roots of $q$ and defining for ourselves a dense open subset of the set of $4$-curvature operators on which there are none. This will turn out to be the set on which $\mathrm{disc_x}(q) \neq 0$.

\begin{lem}\label{lk}
For any given 4-curvature operator $R$, the polynomial   $q(x) = \mathrm{disc_y}(\det(R -xI-yK))$ has no roots introduced by the discriminant function due to the variance of the degree of the $y$ polynomial $\det(R -xI-yK)$ as $x$ varies.
\end{lem}
\begin{proof}
The top coefficient of $\det(R -xI-yK)$, taken as a polynomial in $y$, is $-1$, corresponding to $-y^6$. The coefficient does not depend on $x$, so, taken as a polynomial of $y$,  $\det(R -xI-yK)$ is always of degree $6$, regardless of the value of $x$.
\end{proof}

\begin{lem} \label{mroot}
Let  $R$ be a curvature operator, and $(x_1,y_1)$ a tuple  such that $y_1$ is a multiple root of $f(y) =\det(R -x_1I-yK)$. Denote $p(x,y) =\det(R -(x_1+x)I-(y+y_1)K)$. If, for the given $p$, $a_{10} = 0$ then $x_1$ is a multiple root of  $q(x) = \mathrm{disc_y}(\det(R -xI-yK))$. In particular, if $\mathrm{disc_x}(q)\neq 0$, then for all such pairs  $(x,y)$ we have  $a_{10}\neq 0$.
\end{lem}

\begin{proof}
We may write $p(x,y)= p_{0}(x) + yp_{1}(x) + y^2p_{2}(x) + \cdots  y^5p_{5}(x)  + (-y^6)$, where $p_{n}(x)$, $n\in\{0..5\}$, denote polynomials in $x$. Due to the vanishing coefficients (see Lemma \ref{lem1}), we can write $p_{0}(x)=x^2m_{0}(x)$ and $p_{1}(x)=xm_{1}(x)$. We define $k(x) = \mathrm{disc_y}(p(x,y)) = q(x+x_1)$. We shall now show that $k(x)$ is a multiple of $x^2$, thus $x_1$ is a multiple root of $q$. The discriminant of $p$ can be computed as the determinant of the Sylvester matrix $S$ of $p$ and $\frac{\partial p}{\partial y}$ (see \cite{A1}, chapter 4). We shall examine $S_2$, the submatrix formed from the last two columns of $S$. It takes the form: $$\left(\begin{smallmatrix} 0 && 0\\
.&&.\\
.&&.\\
x^2q_{0}(x) && 0 \\
x q_{1}(x) && x^2q_{0}(x)\\[9pt]
\dotfill&&\dotfill\\[9pt]
 0     &&      0\\
 .&&.\\
 .&&.\\
 x q_{1}(x)  && 0\\
 p_2(x) && x q_{1}(x)
\end{smallmatrix}\right)$$

We note that the following formula for the determinant of $S$:
$$\det(S) = \sum_{a<b} (-1)^{a+b}\det(M_{ab})\det(CM_{ab})$$
where $M_{ab}$ denotes the $2$-minor formed from the $a$ and $b$ rows of $S_{2}$, and $CM_{ab}$ denotes it's complementary minor in $S$. This is an application of Laplace's theorem (\cite{Det}, Sec.77).

We see that for all $2$-minors $M_{ab}$ of $S_2$, we have $\det(M_{ab})$ as a multiple of $x^2$. As such, $\det(S) = \mathrm{disc_y}(p(x,y))$ is a multiple of $x^2$. As  $k(x) = \mathrm{disc_y}(p(x,y)) = q(x+x_1)$, $q$ has a multiple root in $x_1$, thus completing our proof.

\end{proof}

\begin{lem} \label{rel}
 Let $R$ be a $4$-curvature operator, $q(x) = \mathrm{disc_y}(\det(R -xI-yK))$. Assume  $\mathrm{disc_x}(q)\neq 0$. Then for any real $x_1$ such that $q(x_1)=0$, the polynomial  $f(y) = \det(R -x_1I-yK)$ has at least one multiple root which is real.
\end{lem}
\begin{proof}
For a polynomial $f$, $\mathrm{disc}(f) = 0$ if and only if there exists an $y_1$ such that $y_1$ is a multiple root of $f$. In general, $y_1$ does not need to be real. Assume in our particular case that there exists a real $x_1$  such that the set of multiple roots of $f(y) =\det(R -x_1I-yK)$ is non-empty, yet all the multiple roots are non-real. All the coefficients of $f$ are real, therefore all non-real roots come in conjugate pairs. If $f$ has a non-real multiple root, then $f$ has several. Consider a small perturbation $R'$ of $R$ in the space of operators (implicit function theorem). Since $x_1$ is a degree-$1$ root of $q$, it is stable. Therefore there exists a perturbation $x_1'$ such that $f'(y) =\det(R' -x_1'I-yK)$ also has multiple roots. Due to $f'(y)$ being in the neighborhood of $f$, for a small enough perturbation, those roots must also be non-real. Due to the coefficients remaining real, the roots must also come in conjugate pairs. Therefore, if an operator $R$ outside the $ \mathrm{disc}(q)=0$ set has an $x_1$ for which  $f(y) =\det(R -x_1I-yK)$ has only non-real roots, then this property is true for an open neighborhood of $R$ in the space of operators with $ \mathrm{disc}(q) \neq0$. However, when, for a given $x_1$ an $f$ has non-real multiple roots, then it has several multiple roots. This means not only that $\mathrm{disc}(f)$ = 0, but the first subdiscriminant of  $f$ is zero. This means $gcd(\mathrm{disc_y}(\det(R -xI-yK)),\mathrm{subdisc_y}(\det(R -xI-yK)))$ has degree greater than 0. Generically, this does not happen. For the majority of operators, for any given $x_1$ with $q(x_1)=0$, $f(y) =\det(R -x_1I-yK)$ has only 1 double root. Therefore, the set of operators that have an $x_1$ such that $f(y) =\det(R -x_1I-yK)$ has several multiple roots is of measure $0$ in the space of operators. Combining with the previous property, we obtain that all operators $R$ outside of $\mathrm{disc}(q)=0$ have, for every $x_1$ such that $q(x_1)=0$, at least one multiple root of $f(y) =\det(R -x_1I-yK)$ which is real.
\end{proof}

\begin{thm} \label{nt}
Let $R$ be a $4$-curvature operator, $q(x) = \mathrm{disc_y}(\det(R -xI-yK))$. If $\mathrm{disc}(q)\neq 0$ then all the real roots $x_1$ of $q$ correspond to critical points $(x_1, y_1)$ of $R$.
\end{thm}
\begin{proof}
We already know that if $(x_1,y_1)$ is a critical point of $R$ then $q(x_1)=0$ (from Theorem \ref{tx}). For the converse, we combine the previous lemmas:
For $q(x_1)=0$, $\mathrm{disc}(q)\neq 0$, there exists at least one real $y_1$ such that $y_1$ is a multiple root of $f(y) =\det(R -x_1I-yK)$ (from Lemma \ref{rel}). Therefore, for that $y_1$, we have $a_{00}(x_1,y_1) = 0$ and $a_{01}(x_1,y_1) =0$ (these are required conditions for $y_1$ to be a multiple root of $f$). Additionally,  since $\mathrm{disc}(q)\neq 0$, we have $a_{10}(x_1,y_1) \neq 0$ (form Lemma \ref{mroot}). Therefore $y_1$ is real and satisfies  $a_{00} = 0, a_{01} =0, a_{10} \neq 0$. These are sufficient conditions to classify  $(x_1,y_1)$ as a critical point (of kernel dimension 1) (see Definition \ref{d1}).
\end{proof}
Since the $x$ components of critical points correspond to the critical values of $f(v) = vRv$ where $f$ is defined on the set of norm-$1$ decomposable $2$-forms, when $\mathrm{disc}(q) \neq 0$,  the boundaries of the set of real roots of  $q(x) = \mathrm{disc_y}(\det(R -xI-yK))$ give us the maximum and minimum values for the sectional curvature or $R$. These can be checked computationally by applying the signed remainder sequence to $q$.

We recover our {\bf main result} (Theorem \ref{main}) by combining theorems \ref{nt} and \ref{tx}.
\begin{cor}
Given $R$ a $4$-curvature operator on  $\Lambda^2$, the upper and lower bounds of the set of real roots of $q(x) = \mathrm{disc_y}(\det(R-xI-yK))$ are the same as the upper and lower bounds of the sectional curvature or $R$, where $I$ is the identity operator on $2$-forms and $K$ is the $4$-volume form.
\end{cor}

\section{Higher dimensional curvature operators}
Here we briefly discus higher dimensional curvature operators.We examine analogues to the previous theorems and their failures.

As before, we aim to maximise $vRv$ subject to constraints $v \wedge v =0$ and $vIv = 1$, where $I$ is the norm on $2$-forms of dimension $n$ (this is $\widetilde{Gr}(2,n)$, the Grassmannian of oriented $2$-planes in $\mathbb{R}^n$).
Due to the compactness of $\widetilde{Gr}(2,n)$, the critical points give us the bounds for the sectional curvature. The equation $v \wedge v =0$  gives us $\binom{n}{4}$ constraints (see [Thorpe \cite{J},\cite{J2}]), [Singer, \cite{ST}, section 2], and P{\"u}tman  \cite{P}, Section 3.1]).

 Using Lemma $\ref{vform}$, the Lagrangian function is:\\
\begin{eqnarray} L_R(v,x,y_1, y_2,... y_{\binom{n}{4}}) = vRv - x\cdot (vIv-1) - \sum_{p=1}^{\binom{n}{4}} y_p\cdot (vK_{p}v). \label{lr2} \end{eqnarray}
where $K_p$ form a basis on the set of $4$-forms. Its gradient gives us:
\begin{itemize}
\item[1)] $vIv = 1$;
\item[2)] $vK_pv = 0$;
\item[3)] $(R-xI-\sum_{p=1}^{\binom{n}{4}} y_p K_p)v=0$.
\end{itemize}
with $vRv = x$, as in the $4$-dimensional case.
The next step was to reformulate the problem in a way that does not make explicit reference to $v$. We shall only give a partial characterisation of the conditions required for the existence of $v$ satisfying $1)$, $2)$, $3)$. We are looking for $\binom{n}{4} + 1$-uples
$(x,y_1,y_2,...y_{\binom{n}{4}})$ such that there exists a $v$ satisfying the above equations.  Analogously to $n=4$, we shall name these as critical points of $R$ (points $(x_1,y_1,y_2,...y_{\binom{n}{4}})$ such that there exists a $2$-form $v$ making $(v,x,y_1, y_2,... y_{\binom{n}{4}})$ a critical value of $L_R$.\\\\
We shall now give a higher dimensional analogue of Lemma \ref{lem2}:
\begin{lem} \label{nxt}
Assume $q = (x,y_1,y_2,...y_{\binom{n}{4}})$ is a critical point of $R$. \\Define $R' = R-x_{1}I-\sum_{s=1}^{\binom{n}{4}}y_{s}K_s$. \\Let $p(x,z_1,...z_{\binom{n}{4}}) = \det(R' -xI-\sum_{s=1}^{binom{n}{4}}z_{s}K_s)$.  \\
Let $\alpha_0$ be the constant coefficient of the multinomial $p$ and $\alpha_{0s}, s \in \{1,..\binom{n}{4} \} $ be the coefficients corresponding to the monomials ${z_s}^1$.\\
Then $\alpha_0 = 0$ and $\forall s, \alpha_{0s} = 0$.
\end{lem}
\begin{proof}
If $q$ is a critical point, then there exists a $v$ that is simultaneously in the kernel of $R'$, and in the intersection of the null-cones of $K_s$, for all $s$. From the kernel of $R'$ not being null, we get $\alpha_0 = \det(R')=0$. For the rest of the coefficients we shall be selectively applying Lemma \ref{lem2}: if $v$ is in the intersection of the null-cones of $K_s$, then, $\forall s$, the intersection of the kernel of $R'$ and the null-cone of $K_s$ is non-trivial. Selecting a specific $p$, and setting all $z_s, s\neq p$ to $0$, gives us a situation where we can apply Lemma \ref{lem2}, proving $\alpha_{0p} = 0$. This completes the proof.
\end{proof}

In $n=4$, if for a given $x_1$ there exists an $y_1$ making $(x_1,y_1)$ a critical point of $R$, then $y_1$ is a multiple root of $f(y) = \det(R-x_1I-yK)$ (from Theorem \ref{tx}), thus $\mathrm{disc}(f)=0$. Therefore, the set of $x_1$ such that $(x_1,y_1)$ are critical points of $R$ is a subset of the real roots of $q(x) = \mathrm{disc_y}(\det(R-x_1I-yK))$. In higher dimensions, from the previous lemma, we know that
$(x_1,y_1,y_2,...y_{\binom{n}{4}})$ being a critical point of $R$ implies $(x_1,y_1,y_2,...y_{\binom{n}{4}})$ being a null-valued critical point of $f(z_1,...z_{\binom{n}{4}}) = \det(R -x_1I-\sum_{s=1}^{\binom{n}{4}}z_{s}K_s)$. To obtain a similar situation to  \ref{tx}, we require the multi-variable generalization of the discriminant. \cite{MD}.\\

\noindent We obtain the following generalization of Theorem \ref{tx}:\\

\noindent Let $R$ be an $n$-curvature operator, and $\{K_1, \cdots, K_{\binom{n}{4}} \}$ be a basis for the space of $4$-forms of dimension $n$.\\
Denote $p(x,z_1,...z_{\binom{n}{4}}) = \det(R -xI-\sum_{s=1}^{\binom{n}{4}}z_{s}K_s)$. \\ Let $q(x) = \mathrm{disc_z}(p)$ (the multivariate discriminant of $p$ with respect to the $z$ variables).\\ If $(x_1,y_1, \cdots ,y_{\binom{n}{4}})$ is a critical point of $R$, then we have $q(x_1) = 0$. Therefore, if all the real roots of $q$ are positive (nonnegative), then $R$ is sectionally positive (nonnegative).\\

This seems like a good strategy for describing higher dimensional positivity (non-negativity), however, it fails for one important reason:
\begin{thm}
Let $R$ be an $n$-curvature operator, $n\geq 5$, and $\{K_1, \cdots,K_{\binom{n}{4}} \}$ be a basis for the space of $4$-forms of dimension $n$.\\
Denote $p(x,z_1,...z_{\binom{n}{4}}) = \det(R -xI-\sum_{s=1}^{\binom{n}{4}}z_{s}K_s)$. \\ Let $q(x) = \mathrm{disc_z}(p)$. Then $q(x) = 0$.
\end{thm}
\begin{proof}
When $\gamma = (x,y_1,y_2,...y_{\binom{n}{4}})$ is a critical point of $R$,  $\alpha_0 = 0$ and $\forall s, \alpha_{0s} = 0$ (from Lemma \ref{nxt}).
But, due to Lemma \ref{lem2}, this also happens when the rank of  $(R -xI-\sum_{s=1}^{\binom{n}{4}}y_{s}K_s)$ is less than $n-1$. For $n\geq5$, the subset of operators $R$, such that there exists a $4$-form $S$ making $R-S$ of rank $n-2$ is of non-zero measure. We can consider our function   $q_R(x) = \mathrm{disc_z}(p)$ as a polynomial function over the pairs $(x,R)$. Due to the previous property, it will give $0$ on a  non-null set, therefore $0$ everywhere.
\end{proof}

\subsection{Some observations:}
For a given operator $R$, we may describe
$$ p(x,z_1,..., z_{\binom{n}{4}}) = \det(R -xI-\sum_{s=1}^{\binom{n}{4}}z_{s}K_s)$$ in a geometric manner, without resorting to a specific basis. We can define $$p_{R}:(\mathbb{R}, \Lambda^4) \rightarrow \mathbb{R},$$ $$p_{R}(x,W)= \det(R - Ix- W)$$ where $\Lambda^4$ is the space of $4$-forms. Even though for $n\geq5$  the discriminant of $p$ with respect to the coefficients of $W$ is $0$, we suspect a higher-dimensional analogue of \label{thm1} exists, allowing us to characterise the sectional curvature of $R$ from the coefficients of the polynomial $p_{R}$. Studying $p_{R}$ allows the application of algebraic methods to the geometric properties of $R$.

\section{Applications}
Here we apply the previous theorems for particular cases and examine the use of algebraic methods for previously known results.

\begin{defn}
A curvature operator $R$ is known as {\it strongly positive} if there exists a $4$-form $K_n$ such that $R+K_n$ is positive definite \cite{S}.
\end{defn}
All strongly positive curvature operators are sectionally positive. This is known as {\it Thorpe’s trick} (see Thorpe \cite{J},\cite{J2}, Singer and Thorpe \cite{ST}, and P{\"u}ttmann \cite{P} ). For $n=4$, the converse also holds:

\begin{thm}
Let $n=4$ and $R$ be a sectionally positive (nonnegative) curvature operator. Then there exists  $y_1$ such that $R-y_1K$ is positive (semipositive) definite. In other words, in dimension 4, all sectionally positive curvature operators are strongly positive (\cite{S},Proposition 2.2).

\end{thm}
\begin{proof}
We shall study the object $p(x,y) = \det(R- Ix-Ky)$. Having $q(x) = \mathrm{disc_y}(p(x,y))$, assume $ \mathrm{disc_x}(q(x))\neq0$.
For any $y$, $R_y = R-yK$ is symmetric, as such, its characteristic polynomial has exactly $6$ real roots, where $6$ is the rank of $R$. Due to $\mathrm{disc_x}(q(x))\neq0$, the roots of the polynomial are always distinct(see Lemma \ref{mroot}), regardless of $y$. As such, we can consider six parametric functions $\alpha_i(y), i\in\{1..6\}$, corresponding to the ordered roots of the characteristic polynomial of $R-kY$. We shall study $\alpha_1(y)$, corresponding to the smallest root, the smallest eigenvalue of $R-Ky$. The operator $K$ has 3 eigenvalues of  $1$ and 3 eigenvalues of $-1$. \\Therefore $\lim_{y \to -\infty} \alpha_1(y) =-\infty$ and $\lim_{y \to \infty} \alpha_1(y) =-\infty$. Therefore $\alpha_1$ has at least one maximum point. Denote that point as $y_1$, and let $\alpha_1(y_1)=x_1$. Examining the $a_{nm}$ coefficients of $p'(x,y)= \det(R- I(x+x_1)-K(y+y_1))$ we see that  $(x_1,y_1)$ is a critical point to $R$ (see Definition \ref{d1}, $a_{00}=0,a_{01}=0,a_{10}\neq 0$). Therefore, if $R$ is sectionally positive (nonnegative), then $x_1$ is positive (nonnegative). But $x_1$ is also the smallest eigenvalue of $R- Ky_1$, thus completing our proof.  For the operators in the null set of $\mathrm{disc_x}(q(x))=0$, use a limit argument in the set of operators (with prior bounds for $y_1$) to obtain the result.
\end{proof}

Then there exists  $y_1$ such that $R+Ky_1$ is positive (semipositive) definite. In other words, in dimension 4, all sectionally positive curvature operators are strongly positive \cite{S}.

\begin{thm}
Let $n=4$ and $R$ be a sectionally semi-positive curvature operator. Let $Z(R)$ be the set of $V$ such that $VKV=0$ and $VRV=0$. Assume $Z(R)\neq \emptyset$. Then there exists a unique $y_1$ such that  $Z(R)$ is the set of  $V$ such that $VKV=0$ and $V(R-Ky_1)=0$ (see \cite[Theorem 4.1]{J}).
\end{thm}

\begin{proof}
We shall use the same methods as before for the study of the parametric curve $\alpha_1(y)$, giving the smallest eigenvalue of the operator $R-Ky$. As  $\forall $y$,R-Ky$ has the sectional curvature as $R$, the minimum of $R's$ sectional curvature is greater than or equal to $\alpha_1(y)$. Therefore, it is greater than or equal to the maximum value of $\alpha_1(y)$. If $y_1$ is a critical point of $\alpha_1$ then $(\alpha_1(y),y)$ will be a critical point of the sectional curvature (see Definition \ref{d1}). But the maximum of $\alpha_1$  is at most a lower bound of said curvature. Therefore, all critical points of  $\alpha_1$ give the maximum value. Since  $\alpha_1$ has at least one critical point, it has a single critical point, $y_1$. At that point, $\alpha_1(y_1)$'s value is equal to the minimal value of sectional curvature. Let $x_1$ be said value. In our case, $x_1 = 0$. Due to $0$ being the smallest eigenvalue of $R'= R-y_1K$, all vectors $V$ such that $VR'V = 0$ are eigenvectors of $R'$. Therefore, $Z(R)$ is the intersection of the kernel of $R'$ with the set $VKV=0$. For uniqueness, let $V$ be a vector in the kernel of $R'$, such that $VKV=0$. Due to the kernel of $K$ being null, $V$ is not in the kernel of $R'- yK$, $\forall y\neq 0$. This completes the proof.
\end{proof}

\section{Acknowledgements}
I would like to thank my doctoral advisor Ioan Bucataru for guidance in writing this article, Renato Ghini Bettiol for his encouragement and interesting discussions on the topic, and the anonymous reviewer for their thoughtful reading of the work and many helpful suggestions.

\end{document}